\documentclass[10pt]{amsart}

\title[MME for suspension flows]{Measures of maximal entropy for suspension flows over the full shift}

\author{Tamara Kucherenko}\address{Department of Mathematics, The City College of New York, New York, NY, 10031}\email{tkucherenko@ccny.cuny.edu}

\author{Daniel J. Thompson}\address{Department of Mathematics, Ohio State University, Columbus, OH, 43210}\email{thompson.2455@osu.edu}

\date{\today}

\thanks{T.K. is supported by grants from the Simons Foundation \#430032 and from the PSC-CUNY TRADA-47-18. D.T. is supported by NSF grant DMS-$1461163$.}


\newtheorem{thm}{Theorem}
\newtheorem{lem}[thm]{Lemma}
\newtheorem{prop}[thm]{Proposition}

\theoremstyle{definition}

\numberwithin{equation}{section}

\renewcommand{\L}{\Lambda}

\newcommand{\bN}{{\mathbb N}}
\newcommand{\cA}{{\mathcal A}}

\def\Pb{\ifmmode{\Bbb P}\else{$\Bbb P$}\fi}
\def\Z{\ifmmode{\Bbb Z}\else{$\Bbb Z$}\fi}
\def\Q{\ifmmode{\Bbb Q}\else{$\Bbb Q$}\fi}
\def\C{\ifmmode{\Bbb C}\else{$\Bbb C$}\fi}
\def\R{\ifmmode{\Bbb R}\else{$\Bbb R$}\fi}
\def\H{\ifmmode{\Bbb H}\else{$\Bbb H\Bbb N$}\fi}

\def\Susp{\mathrm{ Susp }}
\def\L{\mathcal L}

\def\Ca{\mathcal C}

\usepackage{graphicx}
\usepackage{amsmath,amssymb,amsthm,amsfonts,amscd,flafter,epsf,bm}
\usepackage{mathtools}
\usepackage{graphics}
\usepackage{MnSymbol}
\usepackage{epsfig}
\usepackage{psfrag}
\usepackage{color}
\input xy
\xyoption{all}

\begin{document}

\begin{abstract}

We consider suspension flows with continuous roof function over the full shift $\Sigma$ on a finite alphabet. For any positive entropy subshift of finite type $Y \subset \Sigma$, we explictly construct a roof function such that the measure(s) of maximal entropy for the suspension flow over $\Sigma$ are exactly the lifts of the measure(s) of maximal entropy for $Y$. In the case when $Y$ is transitive, this gives a unique measure of maximal entropy for the flow which is not fully supported. If $Y$ has more than one transitive component, all with the same entropy, this gives explicit examples of suspension flows over the full shift with multiple measures of maximal entropy. This contrasts with the case of a H\"older continuous roof function where it is well known the measure of maximal entropy is unique and fully supported.
\end{abstract}

\maketitle

\setcounter{secnumdepth}{2}

\section{Introduction}
Uniqueness of the measure of maximal entropy (MME) is a central question in the theory of thermodynamic formalism. Many classes of well understood dynamical systems such as Anosov diffeomorphisms or flows, and transitive shifts of finite type, have a unique measure of maximal entropy. It is often a corollary of a uniqueness proof that the MME is fully supported, and this is the case for the systems mentioned above. Finding sufficient conditions for uniqueness of the MME is an active area of research. Recent advances include \cite{BG11, BFSV, rU12, CP16, CT16, GR17}.
Conversely, it is instructive to construct explicit examples with multiple MME, as this can reveal the obstructions to uniqueness. This is the subject of this note. Examples of non-uniqueness of the measure of the maximal entropy for various symbolic dynamical systems include those of Hofbauer \cite{fH77}; Markley and Paul \cite{MP82}; Haydn \cite{nH13}; Savchenko \cite{sS96}; Pavlov \cite{rP16}; Kwietniak, Oprocha and Rams \cite{KOR}; Petersen \cite{kP86} and Krieger \cite{wK74}.

We consider a suspension flow over the full shift $\Sigma$ on a finite alphabet. It is well known that if the roof function $\rho: \Sigma \mapsto (0, \infty)$ is H\"older continuous, then the suspension flow has a unique MME, and this MME is fully supported \cite{PP90}. In contrast, we show that in the class of suspension flows where the roof function $\rho$ is only continuous, the suspension flow may have multiple MME. Even when the MME is unique, it may not be fully supported. These are immediate consequences of our main result, Theorem \ref{main_result}, which states that for any positive entropy subshift of finite type $Y \subset \Sigma$, there exists a continuous function $\rho: \Sigma \mapsto (0, \infty)$ so that the  set of MME for the suspension flow on $\Sigma$ with roof function $\rho$ is exactly the set of lifts of the MMEs for $Y$.
The function $\rho$ is defined explicitly in terms of $Y$, and the result applies equally for suspension semi-flows over the one-sided full shift $\Sigma^+$.
The positive entropy assumption rules out $Y$ being a union of periodic orbits and is necessary (we explain why before Proposition \ref{non-uniqueMME}).

 A point of interest is that any two suspension flows over the same shift $\Sigma$ are orbit equivalent. Thus, our analysis shows that uniqueness of the MME is not preserved by orbit equivalence of flows. In \cite{CLT16}, the existence of an orbit semi-equivalence with a constant height suspension flow was combined with a geometric argument to show that the geodesic flow on a compact locally CAT(-1) space has the weak specification property. This argument yields a unique MME for the flow (among other dynamical results). The results of this paper confirm that orbit equivalence is too weak a property to obtain results such as uniqueness of the MME unless there is additional structure to exploit. This is the expected result, but to the best of our knowledge this phenomenon has not been demonstrated explicitly before.

A dynamical system equipped with a potential for which there is not a unique equilibrium state is said to exhibit a \emph{phase transition}. Examples of suspension flows with phase transitions at the zero potential (i.e. examples for which there does not exist a unique MME) have previously been obtained when the alphabet is infinite by Iommi, Jordan and Todd \cite{IJ13, IJT15}, and when the roof function is allowed to have zeroes by Savchenko \cite{vS98}. In these examples, the phase transition occurs because of \emph{non-existence} of an MME rather than non-uniqueness. To the best of our knowledge, examples of phase transitions in the classical case when the alphabet is finite and the roof function is bounded away from zero do not appear in the literature. In this case, the flow is expansive so the existence of MME is guaranteed \cite{BW}. Phase transitions must therefore arise from the presence of multiple MME rather than the mechanisms of \cite{IJT15, vS98}.

Our method of proof is to reduce the problem to considering equilibrium states for potentials on the base $\Sigma$. This well-known strategy is explained in the monograph of Parry and Pollicott \cite{PP90}. It is convenient to consider the one-sided version of the shifts $Y^+ \subset \Sigma^+$. The problem reduces to finding a continuous function $\rho: \Sigma^+ \to (0, \infty)$ so that the topological pressure $P(-\rho)=0$ and the set of equilibrium states for $\rho$ is the set of MME for $Y^+$. This reduction is explained in \S \ref{s.flows} and the proof of Theorem \ref{main_result}.

Markley and Paul \cite{MP82} have proved the striking related result that if $Y^+ \subset \Sigma^+$ is ANY subshift, then there exists a function $\psi: \Sigma^+ \to \mathbb R$ so that the set of equilibrium states for $\psi$ are the set of measures of maximal entropy for $Y^+$.  Markley and Paul's approach, which relies on Israel's theorem on tangent functionals, does not provide any explicit description of $\psi$. Their result does not apply directly to the suspension flow problem considered here because it is unclear how their function $\psi$, which is not characterized explicitly, could be used as the basis for constructing a roof function $\rho>0$ with $P(-\rho)=0$ and such that the equilibrium states for $\psi$ are the same as those for $-\rho$.  



By contrast, the function $\rho$ that we provide when $Y^+$ is a shift of finite type is completely explicit and we are able to prove the required properties by constructive combinatorial arguments. We guarantee that our roof function $\rho$ is positive by definition, and give explicit pressure estimates to guarantee that $P(-\rho)=0$ and that the equilibrium states for $-\rho$ are the measure(s) of maximal entropy for $Y^+$.

It would be interesting to see how far our finite type hypothesis could be relaxed, and how far it could be relaxed while maintaining an explicit description of $\rho$. It is plausible that the results of this paper hold for any subshift $Y \subset \Sigma$ with positive entropy. However, a different method of proof would be required to establish this, and we suspect that an explicit description of the roof function may no longer be possible at that level of generality.

\section{Preliminaries} \label{s.prelim}
\subsection{Shift spaces.}\label{s.shifts}
We recall some preliminaries about shift spaces (see e.g. \cite{LM95} for more details). In this paper, our analysis takes place in the setting of one-sided shifts, so this is the case we introduce in detail. Let $d\in \bN$ and let $\cA=\{0,\cdots,d-1\}$ be a finite alphabet in $d$ symbols. The \emph{one-sided shift space} $\Sigma^+$ on the alphabet $\cA$ is the set of
all sequences $\xi=(\xi_n)_{n=1}^\infty$ where $\xi_n\in \cA$ for all $n\in \bN$.  We endow $\Sigma^+$ with the Tychonov product topology
which makes $\Sigma^+$ a compact metrizable space. Fixing $\alpha \in (0,1)$, we equip $\Sigma^+$ with the metric
\begin{equation}\label{defmet}
d(\xi,\eta)=\alpha^{\inf\{n\in \bN:\  \xi_n\not=\eta_n\}}
\end{equation}
which induces the Tychonov product topology on $\Sigma^+$.
The shift map $\sigma:\Sigma^+\to \Sigma^+$ defined by $\sigma(\xi)_n=\xi_{n+1}$ is a continuous $d$ to $1$ map on $\Sigma^+$.

For a word $\xi= (\xi_1,\cdots,\xi_n)\in \cA^n$, we denote by $[\xi_1,\cdots,\xi_n]=\{\eta\in \Sigma^+: \eta_1=\xi_1,\cdots, \eta_n=\xi_n\}$ the cylinder generated by $(\xi_1,\cdots,\xi_n)$.
For any two words $\xi=(\xi_1,\cdots,\xi_n)$ with $n\in\bN$ and $\eta=(\eta_1,\cdots,\eta_i)$ with $i\in\bN\cup\{\infty\}$ we denote by $\xi\eta$  their concatenation  $\xi\eta=(\xi_1,\cdots,\xi_n,\eta_1,\cdots,\eta_i)$.

For a continuous function $g:\Sigma^+ \mapsto \R$, the \emph{topological pressure of $g$} is defined by

\begin{equation}\label{Pressure_Def}
P(g)=\lim\limits_{n\to\infty}\frac1n\log \sum_{(\xi_1,...,\xi_n)\in \mathcal{A}^n}\exp{\sup\{S_n g(\eta):\eta\in [\xi_1...\xi_n]\}},
\end{equation}
where $S_n g(\eta)=\sum\limits_{j=0}^{n-1}g(\sigma^j\eta)$. Topological pressure satisfies the variational principle
\begin{equation}\label{varpri}
P(g)=
\sup \left\{h_\mu+\int g\,d\mu\right\},
\end{equation}
where the supremum is taken over all $\sigma$-invariant probability measures and $h_\mu$ denotes the measure-theoretic entropy of $\mu$ (see~\cite{pW82} for details). The measures which realize the supremum are called the \emph{equilibrium states} of $g$.

If $Y\subset \Sigma^+$ is a non-empty closed $\sigma$-invariant set, we  say  that $Y$ is a \emph{subshift}. For a subshift $Y\subset\Sigma^+$ we denote by $\L_n = \L_n(Y)$ the set of all admissible words in $Y$ of length $n$, i.e.
$$\L_n(Y)=\{(\xi_1,\cdots,\xi_n)\in \cA^n:[\xi_1,\cdots,\xi_n]\cap Y\ne\emptyset\}.$$
A \emph{subshift of finite type (SFT)} is a subshift which can be described by a finite set of forbidden words, i.e. words which do not appear in the subshift. We say that an SFT is \emph{$M$-step} if it can be described by a collection of forbidden words of length at most $M+1$. We can assume without loss of generality that every forbidden word for an $M$-step SFT has length exactly $M+1$ by declaring a word of length $M+1$ to be forbidden if it contains a forbidden subword.


A \emph{transition matrix} for a 1-step SFT is a $d\times d$ matrix $A$ with values in $\{0,1\}$ such that $Y=\{\xi\in \Sigma^+: A_{\xi_n \xi_{n+1}}=1 \text{ for all } n \geq 1\}$. By \cite[Proposition 2.3.9]{LM95}, a shift space is a 1-step SFT if and only if it can be described this way (up to renaming symbols). Furthermore, an SFT can always be described by an essential graph \cite[Proposition 2.2.10]{LM95}. This allows us to assume that $A_{ij}=1$ implies that $A_{jk}=1$ for some $k$ in the alphabet $\cA$; roughly the idea is that otherwise we could set $A_{ij}=0$, which does not change $Y$, and repeat this process until we arrive at a new transition matrix with the desired assumption. We assume without loss of generality that the transition matrices $A$ considered in this paper satisfy this property.


Suppose $Y$ is a subshift and $N\in\bN$. We can use the set of words in $\L_{N}(Y)$ as an alphabet, and define the $N$th higher block map $\beta_N$ from $Y$ into the full shift over the alphabet $\L_{N}(Y)$ in the following way. For $\xi\in Y$ the $i$th coordinate of $\beta_N(\xi)$ is a block of coordinates of $\xi$ of length $N$ starting at the position $i$. Then $\beta_N$ is a conjugacy map between $Y$ and the image $\beta_N(Y)$. In particular, any $M$-step SFT is conjugate to a 1-step SFT via the $M$th higher block map. 

The \emph{topological entropy} of a subshift $Y$ is given by $$h(Y)=\lim\limits_{n\to\infty}\frac1n \log   |\L_n(Y)| = \sup \{ h_\mu\},$$ where the supremum is taken over invariant probability measures supported on $Y$.
When $Y$ is a $1$-step SFT with transition matrix $A$, the entropy is $\log \lambda_A$, where $\lambda_A$ is the spectral radius of $A$. It follows from the  Perron-Frobenius theorem that
\begin{equation}\label{Perron-Frobenius}
C_1\lambda_A^n\le |\L_n(Y)|\le C_2\lambda_A^n,
\end{equation}
 for some positive constants $C_1$ and $C_2$.

 A \emph{measure of maximal entropy (MME)} for $Y$ is an invariant measure $\mu$ supported on $Y$ that satisfies $h(Y) = h_\mu$. For a transitive SFT, the transition matrix is irreducible, and there is a unique MME called the \emph{Parry measure} which can be constructed explicitly using the data contained in the transition matrix. When the SFT  $Y$ is not transitive, there are transitive components $Y_1, \ldots Y_n$ corresponding to  irreducible components of the transition matrix. At least one of the transitive SFT's $Y_i$ satisfies $h(Y_i)=h(Y)$. 
 Thus, there is an MME $\mu_i$ for $Y$ supported on each $Y_i$ for which $h(Y_i)=h(Y)$, where $\mu_i$ is the Parry measure for $Y_i$.
The shift of finite type on the alphabet $\{0,1, 2, 3\}$ which is the union of the full shift on $\{0,1\}$ and the full shift on $\{2,3\}$ gives a simple example of an SFT with two MMEs.

The \emph{two-sided shift space} $\Sigma$ is the set of all sequences $\xi=(\xi_n)_{n=-\infty}^\infty$ where $\xi_n\in \cA$ for all $n\in \mathbb Z$. A transition matrix $A$ defines a two-sided 1-step subshift of finite type $Y \subset \Sigma$ by  $Y=\{\xi\in \Sigma: A_{\xi_n\xi_{n+1}}=1 \text{ for all } n \in \mathbb Z\}$ analogously to the one-sided case. 

To define pressure for two-sided shifts, it suffices to use \eqref{Pressure_Def} verbatim, i.e. by summing over cylinder sets $[\xi_1,\cdots,\xi_n]=\{(\eta_n)_{n=-\infty}^{\infty}\in \Sigma: \eta_1=\xi_1,\cdots, \eta_n=\xi_n\}$.
Using this formulation of topological pressure in the $2$-sided case, it is easy to see that given a function $g: \Sigma^+ \to \mathbb R$, there is a natural function $\tilde g: \Sigma \to \mathbb R$ with the same range as $g$ and having the same pressure. We simply set $\tilde{g}(\xi)=g((\xi_n)_{n=1}^\infty)$ for $\xi=(\xi_n)_{n=-\infty}^\infty\in \Sigma$.

\subsection{Suspension Flows.}\label{s.flows}
Let us recall some basic facts about suspension flow that can be found in the book by Parry and Pollicott \cite{PP90}. Let $ \rho:\Sigma\to\mathbb{R}$ be a continuous positive function bounded away from zero.
We define the \emph{suspension space} (relative to $\rho$) as
$$\Susp(\Sigma, \rho) = \{(\xi,s):~ \xi\in\Sigma,\,\, 0\le s\le\rho(\xi)\},$$
where we identify $(\xi,\rho(\xi))=(\sigma\xi,0)$.
The \emph{suspension flow} over $\Sigma$ with roof function $\rho$ is the flow $\Phi = (\varphi_t)_{t\in \mathbb R}$ on $\Susp(\Sigma, \rho)$ defined locally by
$$\varphi_t(\xi,s)=(\xi, s+t)\quad\text{whenever}\quad s+t\in[0,\rho(\xi)].$$ Recall that there is a standard way to lift a $\sigma$-invariant measure $\mu$ on $(\Sigma, \sigma)$ to a $\Phi$-invariant measure $\tilde \mu $ on $\Susp(\Sigma, \rho)$, that every $\Phi$-invariant measure arises this way,  and the Abramov  formula tells us that $$\displaystyle h_{\tilde{\mu}}=\frac{h_\mu}{\int\rho\,d\mu}.$$ Furthermore, it can easily be seen that $\tilde \mu_1 = \tilde \mu_2$ if and only if $\mu_1=\mu_2$. Let $P(-c\rho)$ be the pressure considered in the base. Since the function $c \to P(-c \rho)$ is clearly strictly decreasing and tends to $\infty$ (resp. $-\infty$) as $c \to -\infty$ (resp. $\infty$), there exists a unique $c$ with $P(-c \rho)=0$. Now let $\mu$ be an equilibrium state for $-c\rho$ (this exists by expansivity of $(\Sigma, \sigma)$). Then
\[
0=h_\mu + \int-c\rho d\mu \geq h_\nu + \int-c\rho d\nu
\]
for any $\sigma$-invariant $\nu$ with equality if and only if $\nu$ is an equilibrium state for $-c\rho$. Thus
\[
c=\frac{h_\mu}{\int \rho d \mu} \geq \frac{h_{\nu}}{\int \rho d \nu}.
\]
That is, $c= h_{\tilde \mu} \geq h_{\tilde \nu}$ with equality whenever $\nu$ is an equilibrium state for $-c\rho$. It follows that $c$ is the entropy of the suspension flow on $\Susp(\Sigma, \rho)$, and any measure of maximal entropy for the suspension flow corresponds to an equilibrium state of $-c\rho$ on the base transformation $(\Sigma, \sigma)$.

\subsection{Sub-additive sequences}
Recall that a sequence of real numbers $(b_n)$ is called sub-additive if $b_{i+j}\le b_i+b_j$ for any $i,j\in\bN$. We will use the following property of sub-additive sequences at a technical stage of our analysis. It appears to be a general fact about sub-additive sequences, but we were unable to locate a reference in the literature and hence provide our own proof here.

\begin{lem}\label{sub-additive} Suppose $(b_n)$ is a sub-additive sequence. Then for any $n,k\in\mathbb{N}$ with $k\le n$ we have
$$n[b_1+...+b_{n-1}]+k b_n\ge n b_{1+...+(n-1)+k}$$
\end{lem}
\begin{proof} Clearly, the statement is true for $n=1$. We will proceed by induction. Suppose for any $m<n$ we have
$$m[b_1+...+b_{m-1}]+r b_m\ge m b_{1+...+(m-1)+r},$$
whenever $r\le m$. We will show that then the statement is also true for $n$. Note that for $n=k$ the statement follows immediately from the sub-additivity property of the sequence. We pick any $k< n$ and set $m=n\mod k$. Then $n=lk+m$ for some $l\in\mathbb{N}$. To simplify the notation, let $s=1+...+(n-1)+k$. We rearrange the terms and obtain
\begin{align*}
  n[b_1+...+b_{n-1}]+k b_n & = k[b_1+...+b_{n-1}+b_{n}]+(n-k)[b_1+...+b_{n-1}]\\
  & = k[b_1+...b_{n-k-1}+b_{n-k+1}+b_{n}]+k b_{(n-k)}\\
   &\hspace{5cm}+(n-k)[b_1+...+b_{n-1}]\\
  & \ge k  b_{s}+k b_{(n-k)}+(n-k)[b_1+...+b_{n-1}]
\end{align*}
If $l=1$ then $n-k=m$ and we stop. If not, we apply a similar rearrangement to the terms of $k b_{(n-k)}+(n-k)[b_1+...+b_{n-1}]$.
\begin{align*}
  (n-k)[b_1+...+b_{n-1}]+k b_{(n-k)} &= k[b_1+...+b_{n-1}+b_{n-k}]+(n-2k)[b_1+...+b_{n-1}]\\
  & = k[b_1+...b_{n-2k-1}+b_{n-2k+1}+b_{n-1}]+k b_{(n-2k)}\\
  &\hspace{4cm}+(n-2k)[b_1+...+b_{n-1}]\\
  & \ge k  b_{s}+k b_{(n-2k)}+(n-2k)[b_1+...+b_{n-1}]
\end{align*}
After repeating this process $l$ times we obtain
 \begin{equation}\label{b_n}
 n[b_1+...+b_{n-1}]+k b_n\ge lkb_s+m[b_1+...+b_{n-1}]+kb_m.
 \end{equation}
Since $m<k$, we can set $r=k\mod m$ and write $k=tm+r$. By inductive hypothesis we have $m[b_1+...+b_{m-1}]+r b_m\ge mb_{1+...+(m-1)+r}$. Therefore,
\begin{align*}
  m[b_1+...+b_{n-1}]+kb_m & =m[b_1+...+b_{n-1}+tb_m]+rb_m \\
   & =m[b_1+...+b_{m-1}]+m[b_m+...+b_{n-1}+tb_m]+rb_m\\
   &\ge mb_{1+...+(m-1)+r} + m[b_m+...+b_{n-1}+tb_m]\\
   &\ge mb_{1+...+(n-1)+tm+r}\\
   &= mb_s
\end{align*}
Combining the last inequality with (\ref{b_n}) completes the proof.
\end{proof}

\section{Main Result}
\begin{thm}\label{main_result}
Let $\Sigma$ be the full shift on a finite alphabet, and let $Y \subset \Sigma$ be any positive entropy subshift of finite type. There exists a continuous function $\rho: \Sigma \to (0, \infty)$ so that the  set of MME for the suspension flow on $\Susp(\Sigma, \rho)$ is exactly the set of lifts to $\Susp(\Sigma, \rho)$ of the MMEs for the subshift of finite type $Y$.
\end{thm}
The positive entropy assumption on $Y$ is essential. For an SFT, zero entropy is equivalent to the shift space containing only countably many points. For an irreducible SFT, zero entropy is equivalent to the shift being a finite union of periodic orbits. Given any suspension flow over $\Sigma$, any MME for a zero entropy subshift $Y$ lifts to a zero entropy measure for the flow. 
Since the flow has positive entropy, such measures cannot be MME.

The key ingredient for proving Theorem \ref{main_result} is the following proposition.

\begin{prop}\label{non-uniqueMME}
Suppose $\Sigma^+$ is a one-sided full shift on a finite alphabet and $Y\subset\Sigma^+$ is a $1$-step subshift of finite type with positive entropy $h(Y)$. Let $$a_j=\frac1n\log |\L_n(Y)|+\frac{c}{\sqrt{j}}\quad\text{for}\quad \frac{n(n-1)}{2}\le j< \frac{n(n+1)}{2},$$ where $|\L_n(Y)|$ is the number of words in $\L(Y)$ of length $n$ and $c$ is a constant greater than $\max\{2, h(\Sigma^+)\}$.
Define \begin{equation*}
g(\xi)=\left\{
\begin{aligned}
    -h(Y),&\text{ if }\, \xi\in Y\\
    -a_n,\quad&\text{ if }\, d(\xi, Y)=\alpha^{n+1},\,n\in\bN\\
    -a_1,\quad&\text{ if }\, d(\xi, Y)=\alpha,\\
\end{aligned}
\right.
\end{equation*}
where $\alpha$ is the constant in \eqref{defmet}. Then $g:\Sigma^+\mapsto \R$ is a continuous negative function which is bounded away from zero and has topological pressure $P(g)=0$.
\end{prop}

Assuming this proposition holds, we can prove Theorem \ref{main_result}.

\begin{proof} 
Suppose the alphabet of $\Sigma$ has $d$ symbols, and $Y$ is an $M$-step SFT. We apply the $M$th higher block map $\beta_M$ to the full shift $\Sigma$ and obtain a conjugate subshift $\beta_M(\Sigma)$ in a full shift $\tilde{\Sigma}$ on $d^{M+1}$ symbols. The image of $Y$ under this map is a 1-step SFT, so we can apply Proposition \ref{non-uniqueMME} to the restriction of $\beta_M(Y)$ to the one-sided shift $\tilde{\Sigma}^+$. We denote by $\tilde{Y}$ the image of $Y$ under the map $\beta_M$ and by $\tilde{Y}^+$ its restriction to the one sided shift. Then, there exists a continuous function $g:\tilde{\Sigma}^+\mapsto (-\infty, 0)$ such that $P(g)=0$ and $g(\xi)=-h(\tilde{Y}^+)$ for any $\xi \in \tilde{Y}^+$.

We extend the function $g$ to the two-sided shift $\tilde{\Sigma}$ in a standard way, i.e. for $\xi=(\xi_n)_{n=-\infty}^\infty\in\tilde\Sigma$ we set $\tilde{g}(\xi)=g((\xi_n)_{n=1}^\infty)$. Clearly, $P(\tilde{g})\le 0$. Since the entropies of $\tilde{Y}$ and $\tilde{Y}^+$ are the same, for any measure $\tilde{\mu}$ on $\tilde{Y}$ with maximal entropy we have $$P(\tilde{g})\ge h_{\tilde{\mu}}+\int \tilde{g}\, d\tilde{\mu}=h(\tilde{Y})-h(\tilde{Y})=0.$$ Hence, the measures of maximal entropy of $\tilde{Y}$ are the equilibrium states of $\tilde{g}$.  The restriction $\tilde{g}_0$ of $\tilde{g}$ to $\beta_M(\Sigma)$ still has zero pressure and, since $\tilde{Y}\subset \beta_M(\Sigma)$, any MME for $\tilde{Y}$ is still an equilibrium state for $\tilde{g}_0$.  Now we can use the conjugacy map $\beta_M$ to pull this function back to the original shift space $\Sigma$.

We define the roof function $\rho:\Sigma\mapsto (0,\infty)$ as $\rho(\xi)= -\tilde{g}_0\circ\beta_M(\xi)$. It follows that $P(-\rho)=0$ and thus any measure $\mu$ which is a MME for $Y$  is an equilibrium state for $-\rho$. In view of the discussion in \S \ref{s.flows}, the lift of the measure $\mu$ will be a measure of maximal entropy for the suspension flow on $\Susp(\Sigma, \rho)$.
\end{proof}

The remainder of this section builds up a proof of Proposition \ref{non-uniqueMME}. Let $A$ be a transition matrix for $Y$ as described in \S \ref{s.shifts}. Recall that $|\L_n(Y)|$ is the number of words of length $n$ which appear in the language of $Y$. Since
\begin{equation*}
h(Y)=\lim\limits_{n\to\infty}\frac1n\log |\L_n(Y)|=\inf_{n} \left\{\frac1n\log |\L_n(Y)|\right\},
\end{equation*}
the function $g$ defined above is continuous and $g(\xi)\le -h(Y)$ for all $\xi\in\Sigma$.  Note that if $\mu$ is a measure of maximal entropy for $Y$ then
\[h_\mu+\int g\,d\mu=h(Y)-h(Y)=0,
\]
and thus by the variational principle, $P(g)\geq0$. The main point of our argument is to show that $P(g)\le 0$.

We compute the pressure $P(g)$ using (\ref{Pressure_Def}). Note that for $\xi\in\Sigma$  and $n>1$ we have that $d_\alpha(\xi, Y)=\alpha^{n+1}$ if and only if $ A_{\xi_i\xi_{i+1}}=1$ for $i<n$ and $A_{\xi_n\xi_{n+1}}=0.$ For $\xi\in\Sigma^+$ such that $A_{\xi_1\xi_{2}}=0$ we may have $d_\alpha(\xi, Y)=\alpha^2$ or $d_\alpha(\xi, Y)=\alpha$. The second situation occurs if some symbols of the alphabet $\cA$ do not appear in $Y$.

To handle this detail, we defined $\rho(\xi)=a_1$ for both cases $d_\alpha(\xi, Y)=\alpha^2$ and $d_\alpha(\xi, Y)=\alpha$. Hence, we have
\begin{equation}\label{def g}
g(\xi)=\left\{
\begin{aligned}
    -a_n,\quad&\text{ if }\, A_{\xi_i\xi_{i+1}}=1 \text{ for } i<n \text{ and } A_{\xi_n\xi_{n+1}}=0\\
    -h(Y),&\text{ if }\, A_{\xi_i\xi_{i+1}}=1 \text{ for all } i\\
\end{aligned}
\right.
\end{equation}

Consider a word $(\xi_1,...,\xi_n)\in\cA^n$. If $A_{\xi_i\xi_{i+1}}=1$ for all $i=1,...,n-1$ then $(\xi_1,...,\xi_n)\in \L_n(Y)$ and hence
\begin{equation}\label{supB_n}
\sup\{S_n g(\eta):\eta\in [\xi_1,...,\xi_n]\}=-nh(Y).
\end{equation}

We turn our attention to words that do not belong to $\L_n(Y)$. We have the following lemma.
\begin{lem}
For any word $(\xi_1,...,\xi_n)\notin\L_n(Y)$, we have
\begin{equation}\label{supS_n_final}
\sup\{S_n g(\eta):\eta\in [\xi_1,...,\xi_n]\}\le S_r g(\xi_1,...,\xi_r,\beta_1,\beta_2,...)-(n-r)h(Y),
\end{equation}
where $r=\max\{i: A_{\xi_i\xi_{i+1}}=0\}$ and $\beta\in \Sigma^+$ is any sequence such that $A_{\xi_r\beta_1}=0$.
\end{lem}
\begin{proof}

For any $\eta\in [\xi_1,...,\xi_n]$ we write the Birkhoff sum as
$$S_n g(\eta)=S_r g(\eta)+S_{n-r}g(\sigma^r\eta).$$

Since $A_{\xi_r\xi_{r+1}}=0$ and $r\le n-1$, the distance $d_\alpha(\sigma^i\eta, Y)$ for $1\le i< r$ does not depend on the choice of $\eta\in [\xi_1,...,\xi_n]$. It follows from the definition of $g$ that the sum $S_r g$ is constant on the cylinder $[\xi_1,...,\xi_n]$.

On the other hand, if $r< n-1$ we necessarily have $(\xi_{r+1},...,\xi_{n})\in\L_{n-r}(Y)$. Note that $(\xi_{r+1},...,\xi_{n})\in\L_{n-r}(Y)$ implies that there exist $\beta\in [\xi_{r+1},...,\xi_{n}]\cap Y$. Then the supremum of $S_{n-r}g(\sigma^r\eta)$ over all $\eta\in [\xi_1,...,\xi_n]$ is attained at $\eta=(\xi_1,...,\xi_r,\beta_1,\beta_2,...)$ and is equal to $-rh(Y)$. Combining these two observations we see that
\begin{equation}\label{supS_n}
\sup\{S_n g(\eta):\eta\in [\xi_1,...,\xi_n]\}=S_r g(\xi_1,...,\xi_r,\beta_1,\beta_2,...)-(n-r)h(Y),
\end{equation}
where we may take any $\beta\in [\xi_{r+1},...,\xi_{n}]\cap Y$. Note further, that the value of the sum $S_r g(\xi_1,...,\xi_r,\beta_1,\beta_2,...)$ does not depend on the particular choice of $\xi_{r+1},...,\xi_{n}$ and stays the same as long as $A_{\xi_r\beta_1}=0$.

When $r=n-1$ then the value of the supremum depends on whether or not the symbol $\xi_n$ appears in $Y$. If $\xi_n\in\L_1(Y)$ then the above reasoning applies and the equality (\ref{supS_n}) holds. If $\xi_n\notin\L_1(Y)$ then for $\eta\in [\xi_1,...,\xi_n]$ the sequence $\sigma^{n-1}(\eta)$ starts with $\xi_n$, which implies $d_\alpha(\sigma^{n-1}(\eta), Y)=\alpha$. In this case, $$S_n g(\eta)=S_{n-1}g(\eta)+S_1 g(\sigma^{n-1})(\eta)=S_{n-1}g(\eta)-a_1\le S_{n-1}g(\eta)-h(Y).$$ Again, the value of the sum $S_n g(\eta)$ does not depend on $\eta\in[\xi_1,...,\xi_n]$ and stays the same even if we replace $\xi_n$ by any other symbol $\beta_1\in \cA$, as long as $A_{\xi_{n-1}\beta_1}=0$.
\end{proof}

We fix $n$ and consider all possible words $(\xi_1,...,\xi_n)\in\cA^n$. We
estimate the partition sum
$$Z_n(g)=\sum_{\xi_1,...,\xi_n\in \mathcal{A}}\exp{\sup\{S_n g(\eta):\eta\in [\xi_1,...,\xi_n]\}}.$$
by grouping the words according to the sizes of end-blocks that are admissible in the language of $Y$.
We consider a subset of the alphabet $\cA_0\subset\cA$, $$\cA_0=\{i\in\cA: \text{ there exists } j\in\cA \text{ with } A_{ij}=0 \}.$$ For each $i\in\cA_0$ we select and fix one element $\beta(i)\in\Sigma^+$ such that $A_{i\beta_1(i)}=0$. It follows from (\ref{supB_n}) and (\ref{supS_n_final}) that
\begin{align*}
  Z_n(g) & =\sum_{\mathclap{\hspace{1cm}(\xi_1,...,\xi_n)\in \L_n}}\exp{\sup\{S_n g(\eta):\eta\in [\xi_1,...,\xi_n]\}}+\sum_{\mathclap{\hspace{1cm}(\xi_1,...,\xi_n)\notin \L_n}}\exp{\sup\{S_n g(\eta):\eta\in [\xi_1,...,\xi_n]\}} \\
   & \le |\L_n(Y)|\exp(-nh(Y)) +\sum_{r=1}^{n-1}\quad\sum_{\mathclap{\substack{\hspace{1cm}\xi_1,...,\xi_{r-1}\in \mathcal{A}\\\quad \xi_r\in\mathcal{A}_0\\\hspace{1cm}(\xi_{r+1},...,\xi_n)\in \L_{n-r}(Y)}}}\exp[S_r g(\xi_1,...,\xi_r,\beta(\xi_r))-(n-r)h(Y)]\\
   &=\begin{aligned}[t] |\L_n(Y)|&\exp(-nh(Y)) \\
                       &+\sum_{r=1}^{n-1}|\L_{n-r}(Y)|\exp(-(n-r)h(Y))\,\,\sum_{\mathclap{\substack{\hspace{1cm}\xi_1,...,\xi_{r-1}\in \mathcal{A}\\\quad \xi_r\in\mathcal{A}_0}}}\exp[S_r g(\xi_1,...,\xi_r,\beta(\xi_r))]
     \end{aligned}
   \end{align*}

Our goal is to show that the last summation is bounded by $1$. This is the main technical point of the proof, and it is proved in the following lemma.
\begin{lem}For all $r \in \mathbb N$, the quantity

\begin{equation}\label{Q(r)}
  Q(r):=\sum_{\mathclap{\substack{\hspace{1cm}\xi_1,...,\xi_{r-1}\in \mathcal{A}\\\quad \xi_r\in\mathcal{A}_0}}}\exp[S_r g(\xi_1,...,\xi_r,\beta(\xi_r))]
\end{equation}
satisfies $Q(r)\leq 1$.
\end{lem}
\begin{proof}
We argue recursively. For $r=1$ we have
\begin{equation}\label{Q1}
Q(1)=\sum_{\xi_1\in\cA_0}\exp g(\xi_1,\beta(\xi_1))=|\cA_0|\exp(-a_1)=|\cA_0|\cdot |\L_1(Y)|\exp(-c).
\end{equation}
Since $c>h(\Sigma^+)=\log |\cA|$, the last expression is less than one, so $Q(1)\le 1$. For $r>1$ we split
\begin{equation}\label{Q(r)Step1}
  Q(r) =\sum_{\mathclap{\substack{\quad(\xi_1,\xi_2)\in\L_2\\\hspace{1cm}\xi_3,...,\xi_{r-1}\in \mathcal{A}\\\quad \xi_r\in\mathcal{A}_0}}}\exp[S_r g(\xi_1,...,\xi_r,\beta(\xi_r))]+\sum_{\mathclap{\substack{\quad(\xi_1,\xi_2)\notin\L_2\\\hspace{1cm}\xi_3,...,\xi_{r-1}\in \mathcal{A}\\\quad \xi_r\in\mathcal{A}_0}}}\exp[S_r g(\xi_1,...,\xi_r,\beta(\xi_r))]\\
\end{equation}
When $(\xi_1,\xi_2)\notin\L_2(Y)$, $A_{\xi_1\xi_2}=0$ and hence $g(\xi_1,...,\xi_r,\beta(\xi_r))=-a_1$. We can write $S_r g(\xi_1,...,\xi_r,\beta(\xi_r))=-a_1+S_{r-1} g(\xi_2,...,\xi_r,\beta(\xi_r))$ and estimate the second sum in (\ref{Q(r)Step1}) as
\begin{align*}
\sum_{\mathclap{\substack{\qquad(\xi_1,\xi_2)\notin\L_2\\\hspace{1cm}\xi_3,...,\xi_{r-1}\in \mathcal{A}\\\quad \xi_r\in\mathcal{A}_0}}}\exp[S_r g(\xi_1,...,\xi_r,\beta(\xi_r))]&\le |\cA|\exp(-a_1)\sum_{\mathclap{\substack{\hspace{1cm}\xi_2,...,\xi_{r-1}\in \mathcal{A}\\\quad \xi_r\in\mathcal{A}_0}}}\exp[S_{r-1} g(\xi_2,...,\xi_r,\beta(\xi_r))]\\
&\le |\cA|\exp(-a_1)Q(r-1).
\end{align*}
If $r>2$, we can further split the first sum in (\ref{Q(r)Step1}). For $n>1$ we consider a set of tuples $\Ca_n=\{(\xi_1,...,\xi_n)\in\cA_n:(\xi_1,...,\xi_{n-1})\in\L_{n-1}(Y),\text{ but } (\xi_1,...,\xi_n)\notin\L_n(Y)\}$. Then
\begin{multline}\label{Q(r)Split1}
  \sum_{\mathclap{\substack{\quad(\xi_1,\xi_2)\in\L_2\\\hspace{1cm}\xi_3,...,\xi_{r-1}\in \mathcal{A}\\\quad \xi_r\in\mathcal{A}_0}}}\exp[S_r g(\xi_1,...,\xi_r,\beta(\xi_r))] =\sum_{\mathclap{\substack{\hspace{1cm}(\xi_1,\xi_2,\xi_3)\in\L_3\\\hspace{1cm}\xi_4,...,\xi_{r-1}\in \mathcal{A}\\\quad \xi_r\in\mathcal{A}_0}}}\exp[S_r g(\xi_1,...,\xi_r,\beta(\xi_r))]\\
  +\sum_{\mathclap{\substack{\hspace{1cm}(\xi_1,\xi_2,\xi_3)\in\Ca_3\\\hspace{1cm}\xi_4,...,\xi_{r-1}\in \mathcal{A}\\\quad \xi_r\in\mathcal{A}_0}}}\exp[S_r g(\xi_1,...,\xi_r,\beta(\xi_r))].
\end{multline}
When $(\xi_1,\xi_2,\xi_3)\in\Ca_3$, $A_{\xi_1\xi_2}=1$ and $A_{\xi_2\xi_3}=0$. Therefore, $g(\xi_1,...,\xi_r,\beta(\xi_r))=-a_2$, $g(\Sigma^+(\xi_1,...,\xi_r,\beta(\xi_r)))=-a_1$ and we can partially evaluate $S_r g(\xi_1,...,\xi_r,\beta(\xi_r))$,
$$S_r g(\xi_1,...,\xi_r,\beta(\xi_r))=-a_2-a_1+S_{r-2}g(\xi_3,...,\xi_r,\beta(\xi_r)).$$
As before, we estimate the second sum in (\ref{Q(r)Split1}).
\begin{align*}
\sum_{\mathclap{\substack{\hspace{1cm}(\xi_1,\xi_2,\xi_3)\in\Ca_3\\\hspace{1cm}\xi_4,...,\xi_{r-1}\in \mathcal{A}\\\quad \xi_r\in\mathcal{A}_0}}}\exp[S_r g(\xi_1,...,\xi_r,\beta(\xi_r))]
&\le |\L_2(Y)|\exp(-a_2-a_1)\sum_{\mathclap{\substack{\hspace{1cm}\xi_3,...,\xi_{r-1}\in \mathcal{A}\\\quad \xi_r\in\mathcal{A}_0}}}\exp[S_{r-2} g(\xi_3,...,\xi_r,\beta(\xi_r))]\\
&\le |\L_2(Y)|\exp(-a_2-a_1)Q(r-2).
\end{align*}
Combining the last inequality with (\ref{Q(r)Step1}) and (\ref{Q(r)Split1}) we obtain
\begin{multline}\label{Q(r)Step2}
  Q(r) \le \sum_{\mathclap{\substack{\hspace{1cm}(\xi_1,\xi_2,\xi_3)\in\L_3\\\hspace{1cm}\xi_4,...,\xi_{r-1}\in \mathcal{A}\\\quad \xi_r\in\mathcal{A}_0}}}\exp[S_r g(\xi_1,...,\xi_r,\beta(\xi_r))]\\
  +|\L_2(Y)|\exp(-a_2-a_1)Q(r-2)+ |\cA|\exp(-a_1)Q(r-1).
\end{multline}
If $r>3$, we can split the first sum above in a similar way as in (\ref{Q(r)Split1}). Continuing this process, at step $s<r$  we have
\begin{multline}\label{Q(r)Step_s}
  Q(r) \le \sum_{\mathclap{\substack{\hspace{1cm}(\xi_1,...,\xi_s)\in\L_s\\\hspace{1cm}\xi_{s+1},...,\xi_{r-1}\in \mathcal{A}\\\quad \xi_r\in\mathcal{A}_0}}}\exp[S_r g(\xi_1,...,\xi_r,\beta(\xi_r))]
+|\L_{s-1}(Y)|\exp(-a_{s-1}-...-a_1)Q(r-({s-1}))+...\\
+|\L_2(Y)|\exp(-a_2-a_1)Q(r-2)+ |\cA| \exp(-a_1)Q(r-1).
\end{multline}
We split the first sum
\begin{multline}\label{Q(r)Split_s}
  \sum_{\mathclap{\substack{\hspace{1cm}(\xi_1,...,\xi_s)\in\L_s\\\hspace{1cm}\xi_{s+1},...,\xi_{r-1}\in \mathcal{A}\\\quad \xi_r\in\mathcal{A}_0}}}\exp[S_r g(\xi_1,...,\xi_r,\beta(\xi_r))] =\sum_{\mathclap{\substack{\hspace{1cm}(\xi_1,...,\xi_{s+1})\in\L_{s+1}\\\hspace{1cm}\xi_{s+2},...,\xi_{r-1}\in \mathcal{A}\\\quad \xi_r\in\mathcal{A}_0}}}\exp[S_r g(\xi_1,...,\xi_r,\beta(\xi_r))]\\
  +\sum_{\mathclap{\substack{\hspace{1cm}(\xi_1,...,\xi_{s+1})\in\Ca_{s+1}\\\hspace{1cm}\xi_{s+2},...,\xi_{r-1}\in \mathcal{A}\\\quad \xi_r\in\mathcal{A}_0}}}\exp[S_r g(\xi_1,...,\xi_r,\beta(\xi_r))].
\end{multline}
Since $A_{\xi_i\xi_{i+1}}=1$ for $i<s$ and $A_{\xi_i\xi_{i+1}}=0$ for $i=s$ whenever $(\xi_1,...,\xi_{s+1})\in\Ca_{s+1}$ we can evaluate
$$S_r g(\xi_1,...,\xi_r,\beta(\xi_r))=-a_s-...-a_1+S_{r-s}g(\xi_{s+1},...,\xi_r,\beta(\xi_r))$$
and subsequently estimate the second sum in (\ref{Q(r)Split_s}).
\begin{align*}
\sum_{\mathclap{\substack{\hspace{1.5cm}(\xi_1,...,\xi_{s+1})\in\Ca_{s+1}\\\hspace{1.5cm}\xi_{s+2},...,\xi_{r-1}\in \mathcal{A}\\\qquad \xi_r\in\mathcal{A}_0}}}\exp[S_r g(\xi_1,...,\xi_r,\beta(\xi_r))]
&\le |\L_s(Y)|\exp(-a_s-...-a_1)\sum_{\mathclap{\substack{\hspace{1cm}\xi_{s+1},...,\xi_{r-1}\in \mathcal{A}\\\quad \xi_r\in\mathcal{A}_0}}}\exp[S_{r-s} g(\xi_3,...,\xi_r,\beta(\xi_r))]\\
&\le |\L_s(Y)|\exp(-a_s-...-a_1)Q(r-s).
\end{align*}
The last inequality together with (\ref{Q(r)Split_s}) and \ref{Q(r)Step_s}) imply that
\begin{multline}\label{Q(r)Step_s_complete}
  Q(r) \le \sum_{\mathclap{\substack{\hspace{1cm}(\xi_1,...,\xi_{s+1})\in\L_{s+1}\\\hspace{1cm}\xi_{s+2},...,\xi_{r-1}\in \mathcal{A}\\\quad \xi_r\in\mathcal{A}_0}}}\exp[S_r g(\xi_1,...,\xi_r,\beta(\xi_r))]
+|\L_s(Y)|\exp(-a_{s}-...-a_1)Q(r-s)+...\\
+|\L_2(Y)|\exp(-a_2-a_1)Q(r-2)+|\cA|\exp(-a_1)Q(r-1),
\end{multline}
which completes step $s$.
After $r-1$ steps the inequality (\ref{Q(r)Step_s_complete}) holds with $s=r-1$. In this case  the expression in the first summation is constant. Precisely, for any tuple $(\xi_1,...,\xi_r)\in\L_r$ we have
$A_{\xi_i\xi_{i+1}}=1$ for $i<r$ and $A_{\xi_r\beta(\xi_r)_1}=0$, and hence we can evaluate completely
$$S_r g(\xi_1,...,\xi_r,\beta(\xi_r))=-a_r-...-a_1.$$
We finally arrive at the following
\begin{multline}\label{Q(r)final}
  Q(r) \le |\L_r(Y)|\exp(-a_{r}-...-a_1)
+|\L_{r-1}(Y)|\exp(-a_{r-1}-...-a_1)Q(1)+...\\
+|\L_2(Y)|\exp(-a_2-a_1)Q(r-2)+|\cA|\exp(-a_1)Q(r-1).
\end{multline}

Consider $\sum\limits_{j=1}^s a_j$ with $s\in\bN$. For each $s$ find $n\in\bN$ such that $\frac{n(n-1)}{2}\le s < \frac{n(n+1)}{2}$ and let $k=s-\frac{n(n-1)}{2}$. The sequence $(a_j)$ is defined in such a way that
\begin{equation}\label{sum of a_j}
  \sum\limits_{j=1}^s a_j=\log |\L_1(Y)|+...+\log L_{n-1}+\frac{k}{n}\log |\L_n(Y)|+\sum\limits_{j=1}^s\frac{c}{\sqrt{j}}.
\end{equation}

Note that the first $i$ symbols and the last $j$ symbols of any block in $\L_{i+j}(Y)$ form admissible words of sizes $i$ and $j$ respectively. Since this pair of subwords is determined uniquely by the word of size $(i+j)$, we have $|\L_{i+j}(Y)|\le |\L_i(Y)||\L_j(Y)|$. Therefore, $\log |\L_{i+j}(Y)|\le \log |\L_i(Y)|+\log |\L_j(Y)|$ for any $i,j\in\bN$ and the sequence $(\log |\L_n(Y)|)_n$ is sub-additive. We apply Lemma \ref{sub-additive} to the equation (\ref{sum of a_j}), use the fact that $s=1+...+(n-1)+k$, and obtain
\begin{equation}\label{sum of a_j_2}
  \sum\limits_{j=1}^s a_j\ge \log |\L_s(Y)|+\sum\limits_{j=1}^s\frac{c}{\sqrt{j}}.
\end{equation}
We also estimate $\sum_{j=1}^s\frac{c}{\sqrt{j}}\ge \sum_{j=1}^s\frac{c}{\sqrt{s}}\ge c\sqrt{s}$ and see that
\begin{equation}\label{sum of a_j_final}
  \sum\limits_{j=1}^s a_j\ge \log |\L_s(Y)|+c\sqrt{s}.
\end{equation}

We use (\ref{sum of a_j_final}) to continue with our estimates of $Q(r)$ in (\ref{Q(r)final}).
\begin{align}\label{Q_r continue}
\begin{split}
 Q(r) &\le
 \begin{aligned}[t]
   &|\L_{r}(Y)|\exp(-\log |\L_r(Y)|-c\sqrt{r})\\
   &+|\L_{r-1}(Y)|\exp(-\log |\L_{r-1}(Y)|-c\sqrt{r-1})Q(1)\\
   &\vdots\\
   &+|\L_2(Y)|\exp(-\log |\L_2(Y)|-c\sqrt{2})Q(r-2)\\
   &+|\cA| \exp(-\log |\L_1(Y)|-c)Q(r-1)\\
 \end{aligned}\\
&= \begin{aligned}[t]&\exp(-c\sqrt{r})\\
                &+\exp(-c\sqrt{r-1})Q(1)\\
                &\vdots\\
                &+\exp(-c\sqrt{2})Q(r-2)\\
                &+|\cA| \exp(-\log |\L_1(Y)|-c)Q(r-1).
   \end{aligned}
\end{split}
\end{align}
Recall that $c\ge h(\Sigma^+)=|\cA|$ and thus  $|\cA|\exp(-\log |\L_1(Y)|-c)\le \exp(-\log |\L_1(Y)|)$. On the other hand, in order for the entropy of $Y$ to be positive, it must have at least two admissible symbols, i.e. $|\L_1(Y)|\ge 2$. It follows that the expression in the last line of (\ref{Q_r continue}) is bounded by $\frac12 Q(r-1)$.

We are ready to complete the proof with an induction argument. We know that $Q(1)\le 1$, see (\ref{Q1}). Under the assumption that $Q(s)\le 1$ for all $s<r$, (\ref{Q_r continue}) implies that
\begin{equation}\label{Q_r final2}
  Q(r)\le \sum_{s=2}^{r}\exp(-c\sqrt{s})+\frac12.
\end{equation}
Using the standard integral estimate, one can show that whenever $c\ge 2$ the series in (\ref{Q_r final2}) converges and $\sum\limits_{s=2}^{\infty}\exp(-c\sqrt{s})<\frac12$. Therefore, $Q(r)<1$. The statement of the lemma is now follows by the principle of strong induction. \end{proof}

Recall that
$$Z_n(g)\le |\L_n(Y)|\exp(-nh(Y)) +\sum_{r=1}^{n-1}|\L_{n-r}(Y)|\exp(-(n-r)h(Y))\,\,Q(r).$$
Since $Q(r)<1$, we use (\ref{Perron-Frobenius}) and obtain
\begin{equation*}
  Z_n(g) \le \sum_{r=0}^{n-1}|\L_{n-r}(Y)|\exp(-(n-r)h(Y))\le nC_2
\end{equation*}
It follows that
\begin{equation*}
  P(g)=\lim_{n\to\infty}\frac1n Z_n(g)\le 0,
\end{equation*}
which completes the proof of Proposition \ref{non-uniqueMME}.
\bibliographystyle{plain}
\bibliography{nonuniqueMMEbiblio}

\begin{thebibliography}{10}

\bibitem{BW}
R.~Bowen and P.~Walters.
\newblock Expansive one-parameter flows.
\newblock {\em J. Differential Equations}, 12:180--193, 1972.

\bibitem{BG11}
A.I. Bufetov and B.M. Gurevich.
\newblock {Existence and uniqueness of the measure of maximal entropy for the
  Teichm\"uller flow on the moduli space of Abelian differentials}.
\newblock {\em SB Math}, 202(7):935--970, 2011.

\bibitem{BFSV}
J.~Buzzi, T.~Fisher, M.~Sambarino, and C.~V{\'a}squez.
\newblock Maximal entropy measures for certain partially hyperbolic, derived
  from {A}nosov systems.
\newblock {\em Ergodic Theory Dynam. Systems}, 32(1):63--79, 2012.

\bibitem{CP16}
V.~Climenhaga and R.~Pavlov.
\newblock One-sided almost specification and intrinsic ergodicity.
\newblock {\em Ergodic Theory Dynam. Systems (to appear)}, 2018.
\newblock Published online as Firstview article.

\bibitem{CT16}
Vaughn Climenhaga and Daniel~J. Thompson.
\newblock Unique equilibrium states for flows and homeomorphisms with
  non-uniform structure.
\newblock {\em Adv. Math.}, 303:744--799, 2016.

\bibitem{CLT16}
D.~Constantine, J-F. Lafont, and D.~J. Thompson.
\newblock The weak specification property for geodesic flows on {CAT(-1)}
  spaces.
\newblock {\em Groups, Geometry and Dynamics (to appear)}, 2019.
\newblock Preprint available at arXiv:1606.06253.

\bibitem{GR17}
K.~Gelfert and R.~Ruggiero.
\newblock Geodesic flows modeled by expansive flows.
\newblock {\em Proceedings of Edinburgh Mathematical Society}, 62(1):61--95,
  2019.

\bibitem{nH13}
Nicolai T.~A. Haydn.
\newblock Phase transitions in one-dimensional subshifts.
\newblock {\em Discrete Contin. Dyn. Syst.}, 33(5):1965--1973, 2013.

\bibitem{fH77}
Franz Hofbauer.
\newblock Examples for the nonuniqueness of the equilibrium state.
\newblock {\em Trans. Amer. Math. Soc.}, 228(223--241.), 1977.

\bibitem{IJ13}
Godofredo Iommi and Thomas Jordan.
\newblock Phase transitions for suspension flows.
\newblock {\em Comm. Math. Phys.}, 320(2):475--498, 2013.

\bibitem{IJT15}
Godofredo Iommi, Thomas Jordan, and Mike Todd.
\newblock Recurrence and transience for suspension flows.
\newblock {\em Israel J. Math.}, 209(2):547--592, 2015.

\bibitem{wK74}
Wolfgang Krieger.
\newblock On the uniqueness of the equilibrium state.
\newblock {\em Math. Systems Theory}, 8(2):97--104, 1974/75.

\bibitem{KOR}
Dominik Kwietniak, Piotr Oprocha, and Michal Rams.
\newblock On entropy of dynamical systems with almost specification.
\newblock {\em Israel J. Math.}, 213(1):475--503, 2016.

\bibitem{LM95}
D.~Lind and B.~Marcus.
\newblock {\em An Introduction to symbolic dynamics and coding}.
\newblock Cambridge University Press, 1995.

\bibitem{MP82}
Nelson~G. Markley and Michael~E. Paul.
\newblock Equilibrium states of grid functions.
\newblock {\em Trans. Amer. Math. Soc.}, 274(1):169--191, 1982.

\bibitem{PP90}
W.~Parry and M.~Pollicott.
\newblock {\em Zeta functions and the periodic orbit structure of hyperbolic
  dynamics}.
\newblock Number 187-188 in Ast\'erisque. Soc. Math. France, 1990.

\bibitem{rP16}
Ronnie Pavlov.
\newblock On intrinsic ergodicity and weakenings of the specification property.
\newblock {\em Adv. Math.}, 295:250--270, 2016.

\bibitem{kP86}
Karl Petersen.
\newblock Chains, entropy, coding.
\newblock {\em Ergodic Theory Dynam. Systems}, 6(3):415--448, 1986.

\bibitem{vS98}
S.~V. Savchenko.
\newblock Special flows constructed from countable topological {M}arkov chains.
\newblock {\em Funktsional. Anal. i Prilozhen.}, 32(1):40--53, 96, 1998.

\bibitem{sS96}
Sergey~Valerievich Savchenko.
\newblock Equilibrium states with incomplete supports and periodic
  trajectories.
\newblock {\em Mathematical Notes}, 59(2):163--179, 1996.

\bibitem{rU12}
R.~Ures.
\newblock Intrinsic ergodicity of partially hyperbolic diffeomorphisms with a
  hyperbolic linear part.
\newblock {\em Proc. Amer. Math. Soc.}, 140(6):1973--1985, 2012.

\bibitem{pW82}
P.~Walters.
\newblock {\em An Introduction to Ergodic Theory}, volume~79 of {\em Graduate
  Texts in Mathematics}.
\newblock Springer, New York, 1982.

\end{thebibliography}

\end{document}